\documentclass[12pt,reqno]{amsart}
\usepackage{hyperref,cite}
\usepackage{amssymb,amsthm}
\usepackage{amsfonts,cmtiup,comment,stmaryrd}

\usepackage{mathrsfs,cmtiup}

\makeatletter
\def\blfootnote{\xdef\@thefnmark{}\@footnotetext}
\makeatother

\newtheorem{theorem}{Theorem}[section]
\newtheorem{lemma}[theorem]{Lemma}

\newtheorem{corollary}[theorem]{Corollary}

\theoremstyle{definition}

\newtheorem*{definition*}{Definition}
\numberwithin{equation}{section}

\begin{document}

\title{On finite groups containing an element\\ whose Engel sink is small}

\author{Evgeny Khukhro}
\address{E. I. Khukhro: Charlotte Scott Research Centre for Algebra, University of Lincoln, U.K.}
\email{khukhro@yahoo.co.uk}

\author{Pavel Shumyatsky}

\address{P. Shumyatsky: Department of Mathematics, University of Brasilia, DF~70910-900, Brazil}
\email{pavel@unb.br}

\thanks{The first author was partially supported by the International Center for Mathematics at SUSTech in Shenzhen. The second author was supported by  FAPDF and CNPq.}

\begin{abstract}  For an element $g$ of a group $G$, a right Engel sink of $g$ is a subset of $G$ containing all sufficiently long commutators $[...[[g ,x],x],\dots ,x]$ for all $x\in G$. A left Engel sink of $g$ is a subset of $G$ containing all sufficiently long commutators $[...[[x ,g ],g ],\dots ,g]$  for all $x\in G$. Using  the classification of finite simple groups we prove that if a finite group $G$ has an element $g$ such that $G=[G,g]$, then the order of $G$ is bounded in terms of a right Engel sink of $g$,
as well as in terms of a left Engel sink of $g$.
Earlier Guralnick and Tracey proved this in the case where $g$ is an involution without using the classification.

\end{abstract}

\maketitle

\section{Introduction}
We use the abbreviated notation for Engel words $[a,{}_nb]=[...[[a,b],b],\dots ,b]$, where $b$ is repeated $n$ times.
A \emph{right Engel sink} of an element $h$ of a group $H$ is a
subset $R\subseteq H$ containing all sufficiently long commutators $[h,{}_nx]$, that is, for every $x\in H$  there is a positive integer $r(x,h)$ such that $[h,{}_n x]\in  R$ for any $n\geqslant r(x,h)$.  A left Engel sink of $h\in H$ is a subset $L\subseteq H$ containing all sufficiently long commutators $[x ,{}_n h ]$,  that is, for every $x\in H$  there is a positive integer $l(x,h)$ such that $[x,{}_n h]\in  L$ for any $n\geqslant l(x,h)$. Clearly, in a finite group $H$, every element $h\in H$ has a minimal right Engel sink  ${\mathscr R}(h)$, and a minimal left Engel sink  ${\mathscr L}(h)$. An element $h\in H$ is right Engel precisely when  ${\mathscr R}(h)=\{1\}$, and $h\in H$ is left Engel precisely when  ${\mathscr L}(h)=\{1\}$.

It was proved \cite{khu-shu17, kst, gur-tra} that the (generalized) Fitting height and non-soluble length of the left and right Engel sinks of a given element of a finite group are closely  linked with its `global' position in the (generalized) Fitting series and upper nonsoluble series of the group. These results can be viewed as generalizations of theorems of Baer \cite[12.3.7]{rob}, by which a left Engel element of a finite group belongs to the Fitting subgroup, and a right Engel element belongs to the hypercentre. In particular, Guralnick and Tracey \cite{gur-tra} proved that if $\alpha$ is an automorphism of a finite group $G$ such that $G=[G,\alpha ]$, then
$G=\langle {\mathscr L}(\alpha )\rangle$ (where ${\mathscr L}(\alpha )$ is computed in the semidirect product $G\langle\alpha \rangle$).
Of course, results in terms of automorphisms can also be applied to inner automorphisms, that is, to group elements.

In this paper we consider finite groups $G$ containing an element $g$ such that $G=[G,g]$. We prove that the order of $G$ is bounded in terms of $|{\mathscr R}(g)|$, as well as  in terms of $|{\mathscr L}(g)|$. These results are obtained in more general form for the right and left sinks of an automorphism $\varphi$ such that $G=[G,\varphi]$. If $\varphi$ is the inner automorphism of $G$ induced by conjugation by an element $g$, then in the semidirect product $G\langle\varphi\rangle$  we have
$[x,\varphi ]=[x,g]$ for any $x\in G$. Hence  the right and left sinks ${\mathscr R}_{G\langle\varphi\rangle}(\varphi)$ and ${\mathscr L}_{G\langle\varphi\rangle}(\varphi)$ of $\varphi$ in the semidirect product $G\langle\varphi\rangle$ are equal to ${\mathscr R}(g)$ and ${\mathscr L}(g)$, respectively. The corresponding results about automorphisms are more general, since they apply not only to inner automorphisms.

We first state the results about the right sink.

\begin{theorem}\label{t-r}
Suppose that $\varphi$ is an automorphism of a finite group $G$ such that $G=[G,\varphi ]$. Then the order of $G$ is bounded above in terms of the cardinality of the right Engel sink ${\mathscr R}(\varphi)$ of $\varphi$ in $G\langle\varphi\rangle$.
\end{theorem}

Note that if  ${\mathscr R}(\varphi)=\{1\}$, then $\varphi$ belongs to the hypercentre of $G\langle\varphi\rangle$, and then the condition $G=[G,\varphi ]$ implies that $G=1$. In general, it is not true that $|[G,\varphi ]|$ is bounded in terms of $|{\mathscr R}(\varphi)|$. It is worth mentioning two cases where this is true. First, this is true if $G$ is a simple group. In fact, when $G$ is a non-abelian simple group, the result is somewhat stronger, namely, $|G|$ is bounded in terms of the cardinality of a smaller set ${\mathscr R}_G(\varphi)$, which consists of all sufficiently long commutators $[\varphi ,{}_n x]$ for $x\in G$ (rather than $x\in G\langle\varphi\rangle$).

\begin{theorem}
  \label{t-rsimple}
  Suppose that $\varphi$ is a nontrivial automorphism of a non-abelian finite simple group $G$. Then the order of $G$ is bounded above in terms of the cardinality of the right Engel sink ${\mathscr R}_G(\varphi)$.
\end{theorem}

Another case is that of a coprime automorphism.

\begin{corollary}
  \label{c-r}
  Suppose that $\varphi$ is an automorphism of a finite group $G$ of coprime order: $(|G|,|\varphi|)=1$.  Then the order of $[G,\varphi ]$ is bounded above in terms of the cardinality of the right Engel sink ${\mathscr R}(\varphi)$ of $\varphi$ in $G\langle\varphi\rangle$.
\end{corollary}

We now state the results about the left sink ${\mathscr L}(\varphi)$. Note that for left sinks we have ${\mathscr L}_{G\langle\varphi\rangle}(\varphi)={\mathscr L}_G(\varphi)$.

\begin{theorem}\label{t-l}
Suppose that $\varphi$ is an automorphism of a finite group $G$ such that $G=[G,\varphi ]$. Then the order of $G$ is bounded above in terms of the cardinality of the left Engel sink ${\mathscr L}(\varphi)$.
\end{theorem}

Note that if  ${\mathscr  L}(\varphi)=\{1\}$, then $\varphi$ belongs to the Fitting subgroup of $G\langle\varphi\rangle$, and then the condition $G=[G,\varphi ]$ implies that $G=1$. In general, it is not true that $|[G,\varphi ]|$ is bounded in terms of $|{\mathscr L}(\varphi)|$. This is clearly true when $G$ is a simple group. Another case is that of a coprime automorphism.

\begin{corollary}
  \label{c-l}
Suppose that $\varphi$ is an automorphism of a finite group $G$ of coprime order: $(|G|,|\varphi|)=1$.  Then the order of $[G,\varphi ]$ is bounded above in terms of the cardinality of the left Engel sink ${\mathscr L}(\varphi)$.
\end{corollary}

Another corollary of general nature, which does not require the condition  $G=[G,\varphi ]$, is about the subgroup $\langle {\mathscr L}(\varphi)\rangle$ generated by the left Engel sink.

\begin{corollary}
  \label{c-lr}
Suppose that $\varphi$ is an automorphism of a finite group $G$.  The order of $\langle {\mathscr L}(\varphi)\rangle$ is bounded above
\begin{itemize}
  \item[\rm (a)]  in terms of the cardinality of ${\mathscr L}(\varphi)$, and
  \item[\rm (b)]  in terms of the cardinality of ${\mathscr R}_{G\langle \varphi\rangle}(\varphi)$.
\end{itemize}
\end{corollary}

Part (b) generalizes the well-known fact that a right Engel element in a finite group is necessarily left Engel: for any element $g$ of any finite group  $G$, the cardinality of the left Engel sink ${\mathscr L}(g)$ is bounded above in terms of the cardinality of the right Engel sink  ${\mathscr R}(g)$.

The proofs of our results rely on the classification of finite simple groups. It is worth mentioning that for an involution $\tau \in \mathop{Aut}G$ such that $G=[G,\tau ]$
Guralnick and Tracey \cite[Theorem~1.7]{gur-tra} proved that $|G|$ is bounded in terms of $|{\mathscr L}(\tau)|$ without using the classification. The same result for $|{\mathscr R}(\tau)|$ follows easily from the case of $|{\mathscr L}(\tau)|$; see Corollary~\ref{c-gtright} below for more details.

Other previous results concerning Engel sinks dealt with restrictions imposed on Engel sinks of all elements of a group or a subgroup. Recall that a group $G$ is called an Engel group if all elements of $G$ are left Engel, or equivalently, if all elements of $G$ are right Engel. By the elementary theorem of Zorn \cite[12.3.4]{rob}, an Engel finite group is nilpotent. Much more sophisticated results \cite{med,wi-ze} based on the works of Zelmanov \cite{ze92,ze95,ze17} establish (local) nilpotency of certain types of infinite groups with Engel-type conditions. But in general, say, Engel residually finite groups need not be locally nilpotent, as first shown by Golod's examples~\cite{gol}. Imposing restrictions on the right or left Engel sinks of group elements provides  natural generalizations of Engel groups. In several recent papers \cite{as19, bcs, khu-shu18,khu-shu18a,khu-shu21,khu-shu21a,khu-shu25} it was shown that if a finite, or profinite, or compact group satisfies restrictions on the right or left Engel sinks of all of its elements, then the group is close to be (locally) nilpotent. In other papers restrictions on Engel sinks were imposed on all elements of the fixed-point subgroup of an automorphism \cite{aks19, ass19, khu-shu22, khu-shu22a,khu-shu22b,khu-shu23}. In this paper we  use one of the results  of Khukhro and Shumyatsky \cite[Theorem~3.1]{khu-shu18} by which, if, for a positive integer $m$, all elements of a finite group $G$ have left Engel sinks of cardinality at most $m$, then $G$ has a normal subgroup of order bounded above in terms of $m$ with nilpotent quotient.

We give some key definitions with related observations and recall a few known results  in \S\,\ref{s-prel}. Theorem~\ref{t-rsimple} about the right sink of an automorphism of a non-abelian simple group is proved in \S\,\ref{s-rsimple}. The proof uses the classification of finite simple groups and the existence of cyclic Sylow $r$-subgroups for certain Zsigmondy primes $r$ established by Guralnick, Shareshian, and Woodroofe \cite{gsw}. We also use the fact that cyclic Sylow subgroups in finite simple groups are TI-subgroups proved by Blau \cite{bla}. In \S\,\ref{s-rgeneral} we prove Theorem~\ref{t-r} about the right sink of an element in an arbitrary finite group. Theorem~\ref{t-l} about the left sink is proved in \S\,\ref{s-left}. There we use a result of Guralnick and Tracey \cite[Theorem~1.4]{gur-tra} about generation by the left sink,  Hartley's generalized Brauer--Fowler theorem \cite[Theorem~A]{har}, and a generalization of Baer's theorem proved by Kurdachenko and Subbotin \cite[Theorem~3.2]{kur-sub}.

\section{Preliminaries}\label{s-prel}

Throughout the paper all groups are finite.
We use standard notation for terms of the lower central series $\gamma_1(G)=G$ and $\gamma_{k+1}(G)=[\gamma_{k+1}(G),G]$. The nilpotent residual is denoted by $\gamma_{\infty}(G)=\bigcap_i\gamma_i(G)$. The hypercentre of $G$ is denoted  by $\zeta_{\infty}(G)$.
We recall a well-known fact about nilpotent groups.

\begin{lemma}\label{l-nilporder}
If $G$ is a nilpotent group of class $c$ and $|G/G'|=t$, then $|G|$ is bounded in terms of $t$ and $c$.
\end{lemma}
\begin{proof}
This follows from the fact that the $i$-th factor of the lower central series $\gamma_i(G)/\gamma_{i+1}(G)$ is a homomorphic image of the $i$-th tensor power of $G/G'$; see, for example, \cite[5.2.5]{rob}.
\end{proof}

\subsection*{\bf Left Engel sinks.} Let $h$ be an element of a group $H$, and $K\leqslant H$ an $h$-invariant subgroup. A \emph{left Engel sink} of $h$ in $K$ is a subset $L\subseteq K$ containing all sufficiently long commutators $[x,{}_nh]$ for $x\in K$, that is, for every $x\in K$  there is a positive integer $l(x,h)$ such that $[x,{}_n h]\in  L$ for any $n\geqslant l(x,h)$. The intersection of two such left sinks is again a left sink of $h$ in $K$. Hence there is a minimal left sink of $h$ in $K$ denoted by  ${\mathscr L}_K(h)$. When $K=H$ we write simply ${\mathscr L}(h)$. Speaking of a left Engel sink of an element we shall always mean its smallest left Engel sink.

The  left sink ${\mathscr L}_K(h)$ is $h$-invariant.  Clearly, ${\mathscr L}_{K_1}(h)\subseteq {\mathscr L}_K(h)$ for an $h$-invariant subgroup $K_1\leqslant K$, and the image of ${\mathscr L}_K(h)$ in a quotient group $\bar H=H/N$ is the left Engel sink ${\mathscr L}_{\bar K}(\bar h)$ of the image of $h$ in the image of~$K$.     We shall freely use these properties without special references.

 For every element $x\in K$, either $[x,{}_n h]=1$ for some $n$, or $[x,{}_nh]=[x,{}_{n+m}h]\ne 1$ for some positive integers $m,n$. When $m+n$ is minimal possible in this equation, the elements $[x,{}_ih]$ for $i=n,n+1,\dots , n+m-1$ form a cycle under the mapping $u\to [u,h]$ and this \emph{limit cycle} is contained in ${\mathscr L}_K(h)$. Clearly, ${\mathscr L}_K(h)$ is the union of limit cycles of elements of~$K$. A limit cycle may consist of a single element, but all elements of limit cycles are nontrivial by definition. If $N$ is a normal subgroup, then a limit cycle is either entirely contained in $N$ or has empty intersection with $N$.

 The following theorem was proved by Khukhro and Shumyatsky \cite{khu-shu18}.

\begin{theorem}[{\cite[Theorem~3.1]{khu-shu18}}]\label{t-old-l}
Let $G$ be a finite group, and $m$ a positive integer. Suppose that for every $g\in G$ the cardinality of the left Engel sink ${\mathscr L}(g)$ is at most $m$.
 Then $G$ has a normal subgroup $N$ of order bounded in terms of $m$ such that $G/N$ is nilpotent.
\end{theorem}

\subsection*{\bf Right Engel sinks.} Let $h$ be an element of a group $H$, and $K\leqslant H$ an $h$-invariant subgroup. A \emph{right Engel sink} of $h$ in $K$ is a subset $R\subseteq K$ containing all sufficiently long commutators $[h,{}_nx]$ for $x\in K$, that is, for every $x\in K$  there is a positive integer $r(x,h)$ such that $[h,{}_n x]\in  R$ for any $n\geqslant r(x,h)$. The intersection of two such right sinks is again a right sink of $h$ in $K$. Hence there is a minimal right sink of $h$ in $K$ denoted by  ${\mathscr R}_K(h)$. When $K=H$ we write simply ${\mathscr R}(h)$. Speaking of a right Engel sink of an element we shall always mean its smallest right Engel sink.

 Clearly, ${\mathscr R}_{K_1}(h)\subseteq {\mathscr R}_K(h)$ for an $h$-invariant subgroup $K_1\leqslant K$, and the image of ${\mathscr R}_K(h)$ in a quotient group $\bar H=H/N$ is the right Engel sink ${\mathscr R}_{\bar K}(\bar h)$ of the image of $h$ in the image of $K$. The  right sink ${\mathscr R}_K(h)$ is $h$-invariant.   We shall freely use these properties  without special references.

 For every element $x\in K$, either $[h,{}_n x]=1$ for some $n$, or $[h,{}_nx]=[h,{}_{n+m}x]\ne 1$ for some positive integers $m,n$. When $m+n$ is minimal possible in this equation, the elements $[h,{}_ix]$ for $i=n,n+1,\dots , n+m-1$ form a cycle under the mapping $u\to [u,x]$ and this \emph{limit cycle} is contained in ${\mathscr R}_K(h)$. Clearly, ${\mathscr R}_K(h)$ is the union of limit cycles of elements of~$K$. A limit cycle may consist of a single element, but all elements of limit cycles are nontrivial by definition.  If $N$ is a normal subgroup, then a limit cycle is either entirely contained in $N$ or has empty intersection with $N$.

 \subsection*{\bf Automorphisms.} If $\varphi$ is an automorphism of a group $B$ we use the usual notation for commutators $[b,\varphi ]=b^{-1}b^\varphi$ and commutator subgroups $[B,\varphi ]=\langle [b,\varphi ]\mid b\in B\rangle$, as well as for centralizers $C_B(\varphi )=\{b\in B\mid b^\varphi=b\}$.   The automorphism induced by $\varphi  $ on the quotient by a normal $\varphi $-invariant subgroup is denoted by the same letter~$\varphi $.

 An automorphism $\varphi$ of a finite group $G$ is called a coprime automorphism if  the orders of $\varphi$ and $G$ are coprime: $(|G|,|\varphi|)=1$. A coprime automorphism of a $p$-group is called a $p'$-automorphism. The following lemma collects several well-known properties of coprime automorphisms of finite groups, which will be often used without special references.

\begin{lemma}\label{l-nakr}
Let $\varphi$ be a coprime automorphism of a finite group $G$.
\begin{itemize}
  \item[\rm (a)]  If $N$ is a normal $\varphi$-invariant subgroup of~$G$, then the fixed points of $\varphi$ in  the quotient $G/N$ are covered by the fixed points of $\varphi$ in $G$, that is, $C_{G/N}(\varphi) = C_G(\varphi)N/N$.

  \item[\rm (b)] We have $[[G, \varphi ],\varphi ]=[G,\varphi ]$.

  \item[\rm (c)]   If $G$ is abelian, then $G=[G,\varphi ]\times C_G(\varphi)$.
%    \item[\rm (d)]   If $G$ is a $p$-group, then $\varphi$ acts faithfully on the Frattini quotient $G/\Phi (G)$.????
\end{itemize}
  \end{lemma}

An automorphism of order $p^k$ for some prime $p$ is called a $p$-automorphism.
The following fact is well-known.

\begin{lemma}
\label{l-pp}
If $\alpha $ is a $p$-automorphism of an elementary abelian $p$-group $V$, then $|V|$ is bounded above in terms of $|C_V(\alpha)|$ and $|\alpha|$.
\end{lemma}

\begin{proof}
Consider  $V$ as an $\mathbb{F}_p\langle\alpha\rangle$-module. The dimension of  $|C_V(\alpha)|$ is equal to the number of Jordan blocks in the Jordan normal form of $\alpha$. Each of the blocks has size at most $|\alpha|$. Hence the result.
\end{proof}

Part (c) of the next elementary lemma was essentially proved in \cite[Lemma~2.5]{khu-shu25}.

\begin{lemma}\label{l3}
Suppose that a cyclic group $\langle \alpha \rangle$ acts by
automorphisms on a finite abelian group $V$.
\begin{itemize}

  \item[\rm (a)] The left sink $\mathscr{L}_V(\alpha)=\mathscr{L}_{V\langle\alpha\rangle}(\alpha)$ is a subgroup of $V$.

  \item[\rm (b)] For a power $\alpha^k$ of $\alpha$, we have $\mathscr{L}(\alpha^k)\subseteq \mathscr{L}(\alpha)$.

       \item[\rm (c)] If $V=[V,\alpha]$, then $V=\mathscr{L}(\alpha)$.
\end{itemize}
\end{lemma}

\begin{proof}
(a) For every element $x\in V$ there is a least positive integer $j(x)$ such that
$[x,{}_{j(x)}\alpha ]\in {\mathscr L}( \alpha)$. If $[x,{}_{j(x)}\alpha]\ne 1$, this means that
$[x,{}_{j(x)}\alpha ]$ belongs to a limit cycle in ${\mathscr L}( \alpha)$.
Then $1\ne [x,{}_j\alpha ]\in {\mathscr L}( \alpha)$ for all $j\geqslant j(x)$. Let $j_0=\max\{j(x)\mid x\in
V\}$.
If $1\ne [x,{}_j\alpha ]\in {\mathscr L}( \alpha)$, by adding several times the length of the
limit cycle containing $[x,{}_{j}\alpha ]$ we can assume that $j\geqslant j_0$. Then
$[x,{}_{j}\alpha ]=[[x,{}_{j-j_0}\alpha ],{}_{j_0}\alpha]$. Thus, all elements of ${\mathscr L}(
\alpha)$ have the form $[u,{}_{j_0}\alpha]$,  and all such elements (whether trivial or not) belong to ${\mathscr L}( \alpha)$.
Since $[u,{}_{j_0}\alpha]\cdot [v,{}_{j_0}\alpha]=[uv,{}_{j_0}\alpha]$ in the abelian group $V$, we see that ${\mathscr L}( \alpha)$ is a subgroup.

(b) We apply repeatedly the standard commutator formula $[x,yz]=[x,z][x,y][x,y,z]$. We have $[v,\alpha^k]=[v,\alpha]\cdot [v,\alpha^{k-1}]\cdot [[v,\alpha^{k-1}],\alpha]$, and so on, using the fact that $V$ is abelian. Then $[v,{}_j\alpha^k]$ is a product of commutators of the form $[v,{}_s\alpha]$ for $s\geqslant j$. For big enough $j$ all these commutators belong to ${\mathscr L}( \alpha)$, which is a subgroup by (a). Hence ${\mathscr L}( \alpha^k)\leqslant {\mathscr L}( \alpha)$.

(c) As shown in  \cite[Lemma~2.5(a)]{khu-shu25}], we have  $C_V(\alpha)=1$ and  the mapping $u\mapsto [u,\alpha]$ is an automorphism of $V$.  Thus every $v\in V$ is a commutator of the form $v=[u,\alpha]$, $u\in V$, an therefore of the form $[u,{}_k\alpha]$ for any $k$. Hence $V=\mathscr{L}(\alpha)$.
\end{proof}

Another lemma is from \cite{khu-shu19}.

\begin{lemma}[{\cite[Lemma~2.5]{khu-shu19}}]\label{l-metab}
 If $G$ is a metabelian group, then the minimal right Engel sink of $g\in G$ contains the minimal left Engel sink of the inverse $g^{-1}$, that is, ${\mathscr L}(g^{-1})\subseteq {\mathscr R}(g)$.
\end{lemma}

We use this lemma to show that the result of  Guralnick and Tracey about the left Engel sink of an involution also yields the same kind of result for the right sink. More precisely, they proved in \cite[Theorem~1.7]{gur-tra} without using the classification that if $\tau\in \operatorname{Aut}G$ is an automorphism of order~2 of a finite group $G$ such that $G=[G,\tau]$, then $|G|$ is bounded above in terms of $|J_G(\tau)|$, where $J_G(\tau)=\{g\in G\mid g\text{ has odd order and }g^\tau=g^{-1}\}$.

\begin{corollary}[R. M. Guralnick and G. Tracey]
\label{c-gtright}
 If $\tau\in \operatorname{Aut}G$ is an automorphism of order~2 of a finite group $G$ such that $G=[G,\tau]$, then $|G|$ is bounded above in terms of $|{\mathscr L}_G(\tau)|$ and in terms of $|{\mathscr R}_{G\langle \tau\rangle}(\tau)|$.
\end{corollary}

\begin{proof}
 For $g\in J_G(\tau)$  we have $\langle g\rangle\subseteq  {\mathscr L}_{\langle g\rangle}(\tau)\subseteq {\mathscr L}_{G}(\tau)$. Hence $|G|$ is bounded above in terms of $|{\mathscr L}_{G}(\tau)|$ by \cite[Theorem~1.7]{gur-tra}. Applying Lemma~\ref{l-metab} to the semidirect product $\langle g\rangle\langle \tau\rangle$, we see that  ${\mathscr L}_{\langle g\rangle\langle \tau\rangle}(\tau)\subseteq {\mathscr R}_{\langle g\rangle\langle \tau\rangle}(\tau)\subseteq {\mathscr R}_{G\langle \tau\rangle}(\tau)$. Hence $|G|$ is bounded above in terms of $|{\mathscr R}_{G\langle \tau\rangle}(\tau)|$.
\end{proof}

Henceforth we write, say,  ``$(a,b,\dots)$-bounded'' to abbreviate ``bounded above by some function depending only on the parameters $a,b,\dots$".

\section{Right sink: simple groups}\label{s-rsimple}

In this section we prove Theorem~\ref{t-r} for simple groups in a stronger version given by Theorem~\ref{t-rsimple}.

\begin{proof}[Proof of Theorem~\ref{t-rsimple}] Recall that $\varphi$ is a nontrivial  automorphism of a non-abelian finite simple group~$G$. We need to prove that  $|G|$ is bounded above in terms of the cardinality of the right Engel sink ${\mathscr R}_G(\varphi)$. Note that $G=[G,\varphi ]$, since $G$ is simple.
  It is convenient to fix the number $m=|{\mathscr R}_G(\varphi)|$; our task is to prove that the order $|G|$ is $m$-bounded.
  We can obviously ignore the sporadic simple groups. To deal with the alternating groups and groups of Lie type, we shall make use of a `large' cyclic Sylow subgroup with the help of the following lemma.

  \begin{lemma}
    \label{l-cyc}
    If $S$ is a cyclic Sylow $s$-subgroup of $G$, then $|S|$ is $m$-bounded, namely, $|S|\leqslant (m-1)^2$.
  \end{lemma}

  \begin{proof}
    We shall be using the fact  that a cyclic Sylow subgroup of a finite simple group is a TI-subgroup proved by Blau \cite{bla},  so that $S\cap S^g=1$ if $g\not\in N_G(S)$.

    We can choose $S$ to be not $\varphi$-invariant. Indeed, if $\varphi$ normalized all Sylow $s$-subgroups of~$G$, then $\varphi$ would belong to the kernel of the permutational action of $G\langle\varphi\rangle$ by conjugation on the set of Sylow $s$-subgroups, and then $G=[G,\varphi ]$ would also normalize all Sylow $s$-subgroups, a contradiction.

    Let $x$ be an arbitrary nontrivial element of $S$. We claim that $[\varphi ,{}_kx]\not\in N_G(S)$ for all $k=0,1,2,\dots$. Indeed, first note that $[\varphi ,{}_0x]:= \varphi\not\in N_G(S)$ by our choice of $S$. If there is some $[\varphi ,{}_kx]\in N_G(S)$, we choose this $k$ to be minimal possible with this property. Then $y:=[\varphi, {}_{k-1}x]\not\in N_G(S)$. On the other hand, $[y,x]\in     N_G(S)$. Given that $[y,x]=(x^{-1})^yx$ and $x\in S\leqslant N_G(S)$, we obtain that $(x^{-1})^y\in N_G(S)$. Since $(x^{-1})^y$ is an $s$-element and $S$ is a Sylow $s$-subgroup, we must have $(x^{-1})^y\in S$. As a result, $x^{-1}\in S^{y^{-1}}\cap S$, whence  $S^{y^{-1}}=S$ because $S$ is a TI-subgroup. But then $y\in N_G(S)$, a contradiction.

    Since $[\varphi ,{}_kx]\not\in N_G(S)$ for $x\in S\setminus\{1\}$ and for any $k$, in particular,  the elements $[\varphi, {}_kx]$ are all nontrivial. Therefore for some $k(x)$, all the elements $[\varphi ,{}_kx]$ for $k\geqslant k(x)$ belong to ${\mathscr R}_G(\varphi)$ forming a limit  cycle under the mapping $u\to [u,x]$, and every element in the cycle does not normalize $S$. Let $U$ be the union of these cycles over all $x\in S\setminus\{1\}$. It follows that for every  $x\in S\setminus\{1\}$  there is at least one pair of elements $u,v\in U$ such that $[u,x]=v$ (here the case $u=v$ is also possible). Since $U\subseteq {\mathscr R}_G(\varphi)\setminus \{1\}$, we have $|U|\leqslant m-1$, so that there are at most $(m-1)^2$ such pairs. Therefore, if $|S|>(m-1)^2$, then there are two different elements $x_1,x_2\in S$ for the same pair $u,v\in U$ such that $[u,x_1]=v$ and  $[u,x_2]=v$. Then  $[u,x_1]=[u,x_2]$,  whence $[u,x_1x_2^{-1}]=1$. Since $x_1x_2^{-1}$ is a nontrivial element of $S$, which is a TI-subgroup, $u$ must normalize $S$, contrary to the above. Therefore $|S|\leqslant (m-1)^2$.
      \end{proof}

  \begin{lemma}
    \label{l-ralt}
    If $G$ is isomorphic to an alternating group, then $|G|$ is $m$-bounded.
  \end{lemma}

\begin{proof}
Let $G\cong A_n$; we need to show that $n$ is $m$-bounded. By Bertrand’s Postulate, which was proved by Chebyshev \cite{che}, there is a prime number $p$ satisfying $n/2<p\leqslant n$. Then the Sylow $p$-subgroup of $A_n$ is cyclic of order $p$. By Lemma~\ref{l-cyc} we have $p\leqslant (m-1)^2$, whence $n<2(m-1)^2$.
\end{proof}

It remains to consider simple groups of Lie type. From now on in this section we assume that $G\cong L_d(q)$ is a group of Lie type $L$ of Lie rank $d$ over the finite field $\mathbb{F} _q$ of characteristic~$p$.  The order of $G$ is bounded in terms of the Lie rank $d$ and the cardinality of the field $q=p^k$. First we shall show that the Lie rank is $m$-bounded, then that $p$ is $m$-bounded, and finally that $k$ is $m$-bounded.

To show that the Lie rank $d$ is $m$-bounded, we shall use the fact that every finite simple group contains a cyclic Sylow $s$-subgroup for some prime $s$. This was noted by Brauer \cite{bra} and is implicit in the classification, as well as was explicitly proved by Hiss \cite{his}. Recall that the order of a cyclic Sylow $s$-subgroup in our group $G$ is $m$-bounded by Lemma~\ref{l-cyc}.

\begin{lemma}
  \label{l-lierank}
 If $G\cong L_d(q)$ is a group of Lie type of rank $d$, then $d$ is $m$-bounded, namely, $d\leqslant \max\{8, 2(m-1)^2\}$.
\end{lemma}

\begin{proof}
  The Weyl groups are sections of groups of Lie type. It is known that the Weyl groups of classical Chevalley and Steinberg groups of Lie rank $d$ involve alternating groups $A_d$.  Let $s$ be a prime such that $G$ has a nontrivial cyclic Sylow $s$-subgroup. Then $s\leqslant (m-1)^2$ by Lemma~\ref{l-cyc}. But if $d\geqslant 2s$, then $A_d$ has a non-cyclic Sylow $s$-subgroup. Hence $d<2s\leqslant 2(m-1)^2$. The Lie ranks of all other exceptional groups of Lie type are at most 8.
\end{proof}

Our next goal is to prove that $p$ is $m$-bounded. The proof of this fact is similar to the proof of Lemma~\ref{l-cyc}, with the role of a cyclic Sylow subgroup taken by a subgroup generated by a regular unipotent element. First we recall some definitions and basic facts. A finite simple group $G$ is a group of Lie type in characteristic $p$
if there exist a simple algebraic group $K$ over the algebraically closed
field $\overline{F}_p$, and an endomorphism $\sigma$  of $K$ such that $G$ is generated by the elements of order $p$ in the fixed-point group $K^{\sigma}$. Every such $K$ contains a regular unipotent element, which can be defined as having the least possible dimension of the centralizer (see \cite[Ch.\,4]{hum}). A regular unipotent element is contained in a unique Borel subgroup, and therefore in a unique maximal connected unipotent subgroup. The group $K^{\sigma}$ contains a regular unipotent element (see \cite[\S\,8.12]{hum}), which is also contained in $G$. It follows that  a regular unipotent element in $G$ is contained in a unique Sylow $p$-subgroup of~$G$.

 \begin{lemma}\label{l-rp}
   If $G$ is a simple group of Lie type over a field of characteristic $p$, then $p$ is $m$-bounded, namely, $p\leqslant (m-1)^2+1$.
 \end{lemma}

\begin{proof}
Let $P$ be a Sylow $p$-subgroup of $G$, and $g$ a regular unipotent element of $G$ contained in $P$. We can choose $P$ to be not $\varphi$-invariant. Indeed, if $\varphi$ normalized all Sylow $p$-subgroups of~$G$, then $\varphi$ would belong to the kernel of the permutational action of $G\langle\varphi\rangle$ by conjugation on the set of Sylow $p$-subgroups, and then $G=[G,\varphi ]$ would also normalize all Sylow $p$-subgroups, a contradiction.

Let $x$ be an arbitrary generator of $\langle g\rangle$. We claim that $[\varphi ,{}_kx]\not\in N_G(P)$ for all $k=0,1,2,\dots$. First note that $[\varphi ,{}_0x]:= \varphi\not\in N_G(P)$ by our choice of $P$. If there is some $[\varphi ,{}_kx]\in N_G(P)$, we choose  $k$ to be minimal possible with this property. Then $y:=[\varphi, {}_{k-1}x]\not\in N_G(P)$. On the other hand, $[y,x]\in  N_G(P)$. Given that $[y,x]=(x^{-1})^yx$ and $x\in P\leqslant N_G(P)$, we obtain that $(x^{-1})^y\in N_G(P)$. Since $(x^{-1})^y$ is a $p$-element and $P$ is a Sylow $p$-subgroup, we must have $(x^{-1})^y\in P$. Then $x^{-1}\in P^{y^{-1}}$, so that $g\in P^{y^{-1}}$. Hence  $P^{y^{-1}}=P$ because $P$ is a unique Sylow $p$-subgroup containing $g$. But then $y\in N_G(S)$, a contradiction.

Since $[\varphi ,{}_kx]\not\in N_G(P)$ for any  $k$, in particular,  the elements $[\varphi, {}_kx]$ are all nontrivial. Therefore for some $k(x)$, all the elements $[\varphi ,{}_kx]$ for $k\geqslant k(x)$ belong to ${\mathscr R}_G(\varphi)$ forming a limit cycle under the mapping $u\to [u,x]$, and every element in the cycle does not normalize~$P$. Let $U$ be the union of these cycles over all  generators of $\langle g\rangle$. It follows that for every  generator $x$ of $\langle g\rangle$ there is at least one pair of elements $u,v\in U$ such that $[u,x]=v$ (here the case $u=v$ is also possible). Since $U\subseteq {\mathscr R}_G(\varphi)\setminus\{1\}$, we have $|U|\leqslant m-1$, so that there are at most $(m-1)^2$ such pairs.
If $p-1>(m-1)^2$, then there are two generators $x_1,x_2$ of $\langle g\rangle$ in different cosets of $\langle g^p\rangle$ in $\langle g\rangle$ for the same pair $u,v\in U$ such that $[u,x_1]=v$ and  $[u,x_2]=v$. Then  $[u,x_1]=[u,x_2]$,  whence $[u,x_1x_2^{-1}]=1$. Since $x_1x_2^{-1}$ is also a   generator of $\langle g\rangle$, we  have $g\in \langle x_1x_2^{-1}\rangle=\langle x_1x_2^{-1}\rangle^u\leqslant P^u$. Hence $P^u=P$,  because $P$ is a unique Sylow $p$-subgroup containing $g$. Then $u\in N_G(P)$, a contradiction. Therefore $p-1\leqslant (m-1)^2$.
  \end{proof}

Recall that now $G\cong L_d(q)$ is a group of Lie type $L$ of Lie rank $d$ over a finite field $\mathbb{F} _q$ of characteristic $p$  with $q=p^k$. It remains to prove that $k$ is $m$-bounded. For that we shall use a theorem of Guralnick, Shareshian, and Woodroofe \cite{gsw} that guarantees the existence of a prime divisor of $|G|$ that does not divide the order of any proper parabolic subgroup. In most cases, this is achieved in \cite{gsw} with the use of so-called Zsigmondy primes. Our arguments are  similar to the proof of Proposition~5.5 in \cite{dgms}. First we recall the definitions and relevant results.

Given positive integers $q$ and $e$, a Zsigmondy prime for the pair $(q, e)$ is a prime $r$ dividing $q^e -1$ but not dividing $q^f - 1$ all positive integers $f<e$.
By Zsigmondy's theorem \cite{zsi}, a~Zsigmondy prime for $(q, e)$ exists unless
$e = 6$ and $q = 2$, or $e = 2$ and $q = 2^k - 1$ for some integer $k$. Table~2 in \cite{gsw} lists finite simple groups $G$ of Lie type over a field $\mathbb{F}_q$ for which there is an exponent $e$ such that  a Zsigmondy prime $r$ for the pair $(q,e)$ does not divide the order of any proper parabolic subgroup of $G$. The only groups of Lie type not covered by this table are, in the notation of \cite{gls},   $A_1(p)$ for some Mersenne prime $p$,  or one of the groups $G_2(2)$, $A^+_5(2)$, $A^-_2 (2)$, $A^-_3 (2)$, $B_3(2)$, $C_3(2)$, or $D^+_
4 (2)$. Clearly, these exceptions are irrelevant for our purpose of bounding $k$ in   $q=p^k$, which is the size of the definition field of our group $G$.

\begin{lemma}\label{l-exp}
 If $G$ is a simple group of Lie type over a field $\mathbb{F}_q$ of characteristic $p$ with $q=p^k$, then $k$ is $m$-bounded, namely, $k\leqslant \max\{4,(m-1)^2-1\}$. \end{lemma}

\begin{proof}
 We can assume that $G$ is one of the groups in  \cite[Table~2]{gsw}, and let $e$ be the corresponding exponent. We observe that, since $q^e = p^{ke}$, a Zsigmondy prime for the pair $(p, ke)$ is also a Zsigmondy prime for $(q, e)$ and therefore also divides the order of $G$ but does not divide the order of any proper parabolic subgroup of $G$. We can obviously assume that $k\geqslant 4$, so that $ke>6$ for every exponent $e$ in  \cite[Table~2]{gsw}. Therefore the pair $(p, ke)$ is never an exceptional case of Zsigmondy's theorem. Hence there is a Zsigmondy prime $r$ for $(p, ke)$. Since the multiplicative order of $p$ modulo $r$ is precisely $ke$, we have $ke\leqslant r-1$.

 Let $R$ be a Sylow $r$-subgroup of $G$. We claim that $R$ is cyclic. Note that $r\ne 2$, since either $r$ does not divide $p-1$ for $p\ne 2$, or $r$ divides $p^{ke}-1$ for $p=2$. If $R$ is non-cyclic, then $R$ contains an elementary abelian subgroup $E$ of order $r^2$ (see, for example, \cite[Theorem~5.4.10]{gor}). But then the action of $E$ on the natural $\mathbb{F}_q$-module is reducible (see, for example, \cite[Theorem~3.2.3]{gor}), so that $E$ is contained in a proper parabolic subgroup. Then $r$ divides the order of this parabolic subgroup, contrary to the above.

 Since $R$ is cyclic,  the order of $R$ is $m$-bounded by Lemma~\ref{l-cyc}, in particular, $r\leqslant (m-1)^2$. As a result, $k\leqslant ke\leqslant r-1\leqslant (m-1)^2-1$.
\end{proof}

Thus, for our group $G\cong L_d(q)$ of Lie type $L$ of Lie rank $d$ over a finite field $\mathbb{F} _q$ of characteristic $p$  with $q=p^k$, we have upper bounds in terms of $m$ for $d$ by Lemma~\ref{l-lierank}, for $p$ by Lemma~\ref{l-rp}, and for $k$ by Lemma~\ref{l-exp}. Therefore the order of $G$ is $m$-bounded.
\end{proof}

\section{Right sink: general case}\label{s-rgeneral}

Here we prove Theorem~\ref{t-r} about the right sink  for an arbitrary finite group.

\begin{proof}[Proof of Theorem~\ref{t-r}]
Recall that $\varphi$ is an automorphism of a finite group $G$ such that $G=[G,\varphi ]$. Let $m=|{\mathscr R}(\varphi)|$, where  ${\mathscr R}(\varphi)$ is the minimal right Engel sink of $\varphi$ in $G\langle\varphi\rangle$. We need to prove that $|G|$ is $m$-bounded.
We proceed by induction on $m$. For $m=1$ the element $\varphi$ is hypercentral in $G\langle\varphi\rangle$, so that  $[G,\varphi ]\langle\varphi\rangle$ is nilpotent and therefore $G=[G,\varphi ]=1$.

In general, there are two possibilities. One is when $G$ contains a minimal normal $\varphi$-invariant subgroup $N$ such that $|{\mathscr R}(\bar \varphi)|<m$ in $\bar G\langle\bar\varphi\rangle=G\langle\varphi\rangle/N$. The alternative is that $G$ has no such subgroups.

First we consider the case where $G$ is abelian.

\begin{lemma}\label{l-rabelian}
Suppose that $G$ is abelian. Then $|G|=m$.
\end{lemma}

\begin{proof}
When $G$ is abelian, the condition $G=[G,\varphi ]$ implies that also  $G=[G,\varphi^{-1}]$. Then  $G={\mathscr L}(\varphi^{-1})$ by Lemma~\ref{l3}(c), while ${\mathscr L}(\varphi^{-1})\subseteq {\mathscr R}(\varphi)$ by Lemma~\ref{l-metab}. Hence $|G|\leqslant m$.
\end{proof}

We now consider the case where $G$ itself is a minimal normal $\varphi$-invariant subgroup.

\begin{lemma}\label{l-relementary}
Suppose that $G$ has no proper (nontrivial) normal $\varphi$-invariant subgroups. Then the order of $G$ is $m$-bounded.
\end{lemma}

\begin{proof}
If $G$ is elementary abelian, then  $|G|$ is $m$-bounded by Lemma~\ref{l-rabelian}.

Now let $G=S_1\times\dots\times S_k$ be a direct product of isomorphic non-abelian simple groups transitively permuted by $\varphi$. We regard elements of $G$ as $k$-tuples of elements of $S\cong S_1$ with $x\in S$ in the $j$-th coordinate being $x^{\varphi^{j-1}}$ for $x\in S_1$.   If $k=1$, then $|G|=|S_1|$ is $m$-bounded by Theorem~\ref{t-rsimple}. Let $k\geqslant 2$. For an element $g=(x,y,\dots )$, where $x,y\in S$, we have $[\varphi, g]=(*,x^{-1}y)$ when $k=2$, and $[\varphi, g]=(x,x^{-1}y, y^{-1},\dots )$ when $k\geqslant 3$. In either case the commutator $[\varphi ,{}_ng]$ has  $[x^{-1}y, {}_{n-1}y]$ in its second coordinate. When $x$ varies, the commutators $[x^{-1}y, {}_{n-1}y]$ for big enough $n$ fill the left Engel sink ${\mathscr L}_S(y)$ in $S$.  Since the commutators $[\varphi ,{}_ng]$ for big enough $n$ belong to ${\mathscr R}(\varphi)$, which has cardinality $m$, it follows that $|{\mathscr L}_S(y)|\leqslant m$ for any $y\in S$. Then $|S|$ is $m$-bounded by Theorem~\ref{t-old-l}.

It remains to show that $k$ is $m$-bounded. Clearly, we can assume that $k\geqslant 3$. Then  for $g=(x,y,1,\dots )$, where $x,y\in S$, the commutators $[\varphi ,{}_ng]$ have the form $(1,[x^{-1}y, {}_{n-1}y],1,\dots )$. When $x$ varies, the commutators $[x^{-1}y, {}_{n-1}y]$ for big enough $n$ fill the left Engel sink ${\mathscr L}_S(y)$ in $S$. Since $S$ is simple, $y$ cannot be a left Engel element, so that there are some nontrivial elements $[\varphi ,{}_ng]$ of ${\mathscr R}(\varphi)$ in $S_2$. Their images under powers of $\varphi$ produce non-trivial elements of ${\mathscr R}(\varphi)$ in every $S_i$. Hence $k\leqslant m$.

Thus, the order of $G$ is $m$-bounded.
\end{proof}

Before proceeding with the proof, we make an observation about limit cycles in ${\mathscr R}( \varphi)$. Namely,  for every element $x\in G$ there is a least positive integer $i(x)$ such that $[\varphi,{}_{i(x)}x]\in {\mathscr R}( \varphi)$. If $[\varphi,{}_{i(x)}x]\ne 1$, this means that $[\varphi,{}_{i(x)}x]$ belongs to a limit cycle in ${\mathscr R}( \varphi)$. Then $1\ne [\varphi,{}_jx]\in {\mathscr R}( \varphi)$ for all $j\geqslant i(x)$. Let $i_0=\max\{i(x)\mid x\in G\}$.

\begin{lemma}\label{l-rinduct}
Suppose that $N$ is a (nontrivial) minimal normal $\varphi$-invariant subgroup of $G$ such that $|{\mathscr R}_{G\langle\varphi\rangle/N}( \varphi)|<m$. Then the order of $G$ is $m$-bounded.
\end{lemma}

\begin{proof}
If $N=G$, then $|G|$ is $m$-bounded  by Lemma~\ref{l-relementary}. Therefore we can assume that $N<G$. By the induction hypothesis, the order of $G/N$ is $m$-bounded. We need to show that the order of $N$ is also $m$-bounded.

 First suppose that $N$ is a direct product of non-abelian simple groups. Then $[N,\varphi ]\ne 1$, as otherwise $G=[G,\varphi ]$ would centralize $N$, a contradiction. The normal subgroup $[N,\varphi ]$ of $N$ is a product of some of the simple factors of $N$. Note that $C_N([N,\varphi ])$ is $\varphi$-invariant and is the product of the simple factors of $N$ that are not in $[N,\varphi ]$. By the same argument, $[N,\varphi ]$ is a direct product of several (possibly one) minimal normal $\varphi$-invariant subgroups of $N$, and if $M$ is one of these  subgroups, then $[M,\varphi ]=M$. Then $|M|$ is $m$-bounded  by Lemma~\ref{l-relementary}.  Since  $M$ is normal in $N\langle\varphi\rangle$, there are at most $|G/N|$ conjugates of $M$. Since $|G/N|$ is $m$-bounded,   the normal closure of $M$ in $G\langle\varphi\rangle$, which is $N$, is of $m$-bounded order, as required.

 Now let $N$ be abelian. If $N$ is central in $G$, then the order of the derived subgroup $G'$ is $m$-bounded by Schur's theorem \cite[Theorem~4.12]{rob1}. The order of the abelian quotient group $\bar G =G/G'$ is at most $m$ by Lemma~\ref{l-rabelian}. As a result, $|G|$ is $m$-bounded.

 Thus we can assume that $N$ is not central in $G$. Note that then ${\mathscr R}( \varphi)\not\subseteq C_G(N)$, as otherwise $\varphi$ would be hypercentral in $G\langle\varphi\rangle/C_G(N)$, so that we would have  $G=[G,\varphi ]=C_G(N)$ and $N$ would be central in $G$, contrary to our assumption.

 Let $a\in {\mathscr R}( \varphi)\setminus C_G(N)$. There is $x\in G$ and a positive integer $i$ such that $a=[\varphi ,{}_ix]$. Adding if necessary several times the length of the limit cycle of  ${\mathscr R}( \varphi)$ containing $a$, we can assume that $i\geqslant i_0$ (where $i_0$ is defined before Lemma~\ref{l-rinduct}).
 Then for any $n\in N$ we have $[\varphi ,{}_i(nx)]=n_1a$ for some $n_1\in N$ and the commutator $[\varphi ,{}_i(nx)]$ also belongs to
 ${\mathscr R}( \varphi)\setminus C_G(N)$. Moreover,
 the commutators $b=[\varphi ,{}_{i+1}x]$ and $[\varphi ,{}_{i+1}(nx)]$ also  belong to ${\mathscr R}( \varphi)\setminus C_G(N)$ and $[\varphi ,{}_{i+1}(nx)]=[n_1a, nx]=bn_2$ for some $n_2\in N$. Using the fact that $N$ is abelian, we rewrite the last equation as
 $$
 [a,n]^x=[x,a]\cdot [x,n_1]^a\cdot bn_2.
 $$
 For given $a,b$ there are at most $m-1$ values for each of the elements $n_1,n_2$ for the products $an_1,bn_2$ to be in  ${\mathscr R}( \varphi)\setminus C_G(N)$. Hence there are at most $(m-1)^2$ possibilities for the value of the commutator $[a,n]$ with $n\in N$. This means that $|N:C_N(a)|\leqslant (m-1)^2$.  Since $|G/N|$ is $m$-bounded and $N$ is abelian, the normalizer of $C_N(a)$ in $G$ has $m$-bounded index. The automorphism $\varphi$ permutes the elements of ${\mathscr R}( \varphi)$ and therefore their centralizers. Since $|{\mathscr R}( \varphi)|=m$,  the normalizer of $C_N(a)$ in $\langle\varphi\rangle$ also has $m$-bounded index. As a result, the number of conjugates of $C_N(a)$ in $G\langle\varphi\rangle$ is $m$-bounded, so that their  intersection has $m$-bounded index in $N$. Since $C_N(a)\ne N$ and $N$ is a minimal normal $\varphi$-invariant subgroup, this intersection must be trivial, so that $|N|$ is $m$-bounded, as required.
 \end{proof}

\begin{lemma}
  \label{l-rlast}
  Suppose that $G$ has no (nontrivial) minimal normal $\varphi$-invariant subgroups of $G$ such that $|{\mathscr R}_{G\langle\varphi\rangle/N}(\varphi)|<m$. Then the order of $G$ is $m$-bounded.
\end{lemma}

\begin{proof}
In particular, $G$ itself is not a minimal normal $\varphi$-invariant subgroup. Hence $G$ has proper nontrivial normal $\varphi$-invariant subgroups $N$ such that $|{\mathscr R}_{G\langle\varphi\rangle/N}( \varphi)|=m$, and let $H$ be a maximal subgroup with these properties. Then $G/H$ must have a  minimal normal $\varphi$-invariant subgroup $N/H$ such that $|{\mathscr R}_{G\langle\varphi\rangle/N}( \varphi)|<m$. Therefore $|G/H|$ is $m$-bounded by Lemma~\ref{l-rinduct}.

Since ${\mathscr R}(\varphi)$ has the same number of elements $m$ as its image ${\mathscr R}_{G\langle\varphi\rangle/N}(\varphi)$ in $G/H$,  every element   $u\in {\mathscr R}( \varphi)$ is the only element of ${\mathscr R}(\varphi)$ in the coset $Hu$.

Let $a\in {\mathscr R}( \varphi)$. There is $x\in G$ such that $[\varphi,{}_ix]=a$ for some $i$. Adding if necessary several times the length of the limit cycle of  ${\mathscr R}( \varphi)$ containing $a$, we can assume that $i\geqslant i_0$ (where $i_0$ is defined before Lemma~\ref{l-rinduct}). Then all the commutators $[\varphi,{}_i(hx)]$ for $h\in H$ belong to ${\mathscr R}(\varphi)$. Since they belong to the same coset of $H$ as $a$, we actually have  $[\varphi,{}_i(hx)]=a$ for all $h\in H$. Similarly, $[\varphi,{}_{i+1}(hx)]$ belongs to ${\mathscr R}(\varphi)$ for all $h\in H$, and therefore $[\varphi,{}_{i+1}(hx)]=[\varphi,{}_{i+1}x]$. As a result,
$$
[\varphi,{}_{i+1}(hx)]=[[\varphi,{}_{i}(hx)], hx]=[\varphi,{}_{i+1}x]=[a,hx]=[a,x],
$$
whence $[a,h]=1$ for all $h$. Thus, ${\mathscr R}(\varphi)$ centralizes $H$.

The condition $G=[G,\varphi ]$ implies that $G$ is the normal closure of ${\mathscr R}(\varphi)$. Indeed, in the quotient of $G\langle\varphi\rangle$ by this normal closure $\varphi$ is hypercentral, and therefore the image of $[G,\varphi ]\langle\varphi\rangle$ is nilpotent, whence the image of $G=[G,\varphi ]$ is trivial.

Since $H$ is normal, all conjugates of ${\mathscr R}(\varphi)$ also centralize $H$; as a result, $H$ is central in~$G$. Since $|G/H|$ is $m$-bounded, the derived subgroup $G'$ has $m$-bounded order by Schur's theorem \cite[Theorem~4.12]{rob1}. The abelian quotient $G/G'$ has order at most $m$ by Lemma~\ref{l-rabelian}. Hence $|G|$ is $m$-bounded.
\end{proof}

The proof of Theorem~\ref{t-r} is complete.
\end{proof}

\section{Left sink}\label{s-left}

Here we prove Theorem~\ref{t-l} about the left sink.

\begin{proof}[Proof of Theorem~\ref{t-l}]
Recall that $\varphi$ is an automorphism of a finite group $G$ such that $G=[G,\varphi ]$. We need to prove that $|G|$ is bounded above in terms of $m=|{\mathscr L}(\varphi)|$. By a result of Guralnick and Tracey \cite[Theorem~1.4]{gur-tra} we have $G=\langle {\mathscr L}(\varphi)\rangle$.

We now prove that the hypercentre of $G$ has $m$-bounded index.

\begin{lemma}
  \label{l-lhyp}
  If $Z(G)=1$, then  $|G|$ is $m$-bounded.
\end{lemma}

\begin{proof}
We proceed by induction on $m$. If $m=1$, then $\varphi$ is a left Engel element of $G\langle\varphi\rangle$ and therefore belongs to the Fitting subgroup, so that $[G,\varphi ]\langle\varphi\rangle$ is nilpotent and therefore $G=[G,\varphi ]=1$.

Since $G=\langle {\mathscr L}(\varphi)\rangle$ and $Z(G)=1$, the centralizer $C_G(\varphi )$ acts faithfully on $ {\mathscr L}(\varphi)$, and so does $\langle\varphi\rangle$. Hence both $|\varphi|$ and $|C_G(\varphi )|$ are $m$-bounded. Then the quotient by the soluble radical $G/S(G)$ has $m$-bounded order by Hartley's theorem \cite[Theorem~A]{har}.

If $S(G)=1$, then we are done. Otherwise we choose a soluble
minimal normal $\varphi$-invariant subgroup $N$, which is an elementary $p$-group for some prime $p$. We claim that $|N|$ is $m$-bounded. Indeed, let  $\langle\varphi_p\rangle$ and $\langle\varphi_{p'}\rangle$ be the Sylow $p$-subgroup  and  the Hall $p'$-subgroup  of  $\langle\varphi\rangle$, respectively (either of these subgroups can be trivial). Then
$$
C_N(\varphi_p)=C_N(\varphi)\times [C_N(\varphi _p),\varphi _{p'}],
$$
since $\varphi_{p'}$ is a coprime automorphism of $C_N(\varphi_p)$. (This equation also holds if $\varphi_{p'}=1$ or $\varphi_p=1$.) We know that $|C_N(\varphi)|$ is $m$-bounded,  and $|[C_N(\varphi _p),\varphi _{p'}]|$ is $m$-bounded because $[C_N(\varphi _p),\varphi _{p'}]\subseteq {\mathscr L}(\varphi)$ by Lemma~\ref{l-nakr}(b) and Lemma~\ref{l3}(c). Hence $|C_N(\varphi_p)|$ is $m$-bounded. Since
$|N|$ is bounded in terms of $|C_N(\varphi_p)|$ and $|\varphi|$ by Lemma~\ref{l-pp}, and $|\varphi_p|$ is also $m$-bounded, we obtain that $|N|$ is $m$-bounded.

It remains to show that $|G/N|$ is $m$-bounded. We claim that  $|{\mathscr L}_{G/N}(\varphi)|<m$, and then the result will follow by induction. We argue by contradiction, assuming the opposite, that $|{\mathscr L}_{G/N}(\varphi)|=m$. Then every element   $u\in {\mathscr L}( \varphi)$ is the only element of ${\mathscr L}(\varphi)$ in the coset~$Nu$. Furthermore, $[N,\varphi _{p'}]=1$, since $[N,\varphi _{p'}]\subseteq {\mathscr L}(\varphi_{p'})$ by Lemma~\ref{l-nakr}(b) and Lemma~\ref{l3}(c), while ${\mathscr L}(\varphi_{p'})\subseteq {\mathscr L}(\varphi)$ by Lemma~\ref{l3}(b), since $\varphi_{p'}$ is a power of $\varphi$. Therefore $C_N(\varphi)=C_N(\varphi _p)\ne 1$.

For every element $x\in G$ there is a least positive integer $j(x)$ such that $[x,{}_{j(x)}\varphi ]\in {\mathscr L}( \varphi)$. If $[x,{}_{j(x)}\varphi ]\ne 1$, this means that $[x,{}_{j(x)}\varphi ]$ belongs to a limit cycle in ${\mathscr L}( \varphi)$. Then $1\ne [x,{}_j\varphi ]\in {\mathscr L}( \varphi)$ for all $j\geqslant j(x)$.
Let $j_0=\max\{j(x)\mid x\in G\}$.

Every element  $a\in {\mathscr L}(\varphi)$ has the form $a= [x,{}_k\varphi ]$ for some $x\in G$ and some positive integer~$k$. Adding to $k$ if necessary several times the length of the limit cycle in ${\mathscr L}(\varphi)$ containing $a$, we can assume that $k\geqslant j_0$. Now, for every $c\in C_N(\varphi)$ we have
$$
[x,{}_k\varphi ]^c=[x^c,{}_k\varphi^c] =[x^c,{}_k\varphi ] =[x,{}_k\varphi ],
$$
since both $[x^c,{}_k\varphi ]$ and $[x,{}_k\varphi ]$ are elements of ${\mathscr L}(\varphi)$ in the same coset of $N$. Thus, $a^c=a$ for every $a\in {\mathscr L}(\varphi)$ and every $c\in C_N(\varphi)$, which means that $C_N(\varphi)$ centralizes ${\mathscr L}(\varphi)$ and therefore is contained in the centre of $G=\langle{\mathscr L}(\varphi)\rangle$. Since $C_N(\varphi)\ne 1$, this contradicts the hypothesis that $Z(G)=1$.

 Thus, $|{\mathscr L}_{G/N}(\varphi)|<m$, whence $|G/N|$ is $m$-bounded by induction. Since $|N|$ is $m$-bounded, the order $|G|$ is $m$-bounded.
 \end{proof}

We return to the proof of Theorem~\ref{t-l}. By Lemma~\ref{l-lhyp} the quotient by the hypercentre $G/\zeta_{\infty}(G)$ has $m$-bounded order. By a generalization of Baer's theorem proved by Kurdachenko and Subbotin \cite[Theorem~3.2]{kur-sub}, the order of the nilpotent residual $\gamma_{\infty}(G)$ is bounded above in terms of $|G/\zeta_{\infty}(G)|$. Thus, $|\gamma _{\infty}(G)|$ is $m$-bounded. Therefore it remains to prove that $|G/\gamma _{\infty}(G)|$ is $m$-bounded. In other words, we can assume that the group $G$ is nilpotent. Let $G=P_1\times\cdots \times P_k$ be the decomposition of $G$ into a direct product of its Sylow subgroups $P_i$. The condition $G=[G,\varphi ]$ implies that $[P_i,\varphi ]=P_i$ for every $i$. In particular, $ {\mathscr L}_{P_i}(\varphi)\ne \{1\}$ for every $i$. Hence the number of Sylow subgroups is at most $m$. Thus we can assume that henceforth $G$ is a finite $p$-group.

Let $\langle\varphi_{p'}\rangle$ be the Hall $p'$-subgroup of $\langle\varphi\rangle$.  For the Frattini quotient $V=G/\Phi(G)$ of $G$, we have $[V,\varphi_{p'}]=V$, since otherwise $\varphi$ acts as a $p$-automorphism on a nontrivial $p$-group $V/[V,\varphi_{p'} ]$, whence
$[V/[V,\varphi_{p'}],\varphi ]\ne V/[V,\varphi_{p'}]$ contrary to $V=[V,\varphi ]$. Hence, $[G,\varphi_{p'}]=G$.

\begin{lemma}
  \label{l-lclass}
  The nilpotency class of $G$ is $m$-bounded.
\end{lemma}

\begin{proof}
  We use induction on $m$. If $m=1$, then $\varphi$ belongs to the Fitting subgroup of $G\langle\varphi\rangle$ and therefore $G=[G,\varphi ]=1$.

Let $c$ be the nilpotency class of $G$. We can assume that $c>2$. Note that then $\gamma_{c-1}(G)$ is abelian, as $[\gamma_{c-1}(G),\gamma_{c-1}(G)]\leqslant \gamma_{2c-2}(G)=1$, as $2c-2>c$. We claim that $[\gamma_{c-1}(G),\varphi_{p'}]\ne\nobreak 1$. Otherwise, if $\varphi_{p'}$ centralizes $\gamma_{c-1}(G)$, the group $G=[G,\varphi_{p'}]$ also centralizes $\gamma_{c-1}(G)$, contrary to $[\gamma_{c-1}(G),G]=\gamma_{c}(G)\ne 1$. Since $\varphi_{p'}$ is a coprime automorphism of the abelian group $\gamma_{c-1}(G)$, we have $[\gamma_{c-1}(G),\varphi_{p'}]\subseteq {\mathscr L}_{\gamma_{c-1}(G)}(\varphi_{p'})$ by Lemmas~\ref{l-nakr}(b) and~\ref{l3}(c), and ${\mathscr L}_{\gamma_{c-1}(G)}(\varphi_{p'})\subseteq {\mathscr L}_{G}(\varphi)$ by Lemma~\ref{l3}(b), since $\varphi_{p'}$ is a power of $\varphi$. Since $ [\gamma_{c-1}(G),\varphi_{p'}]\ne 1$, we thus obtain that $|{\mathscr L}_{G/\gamma_{c-1}(G)}(\varphi)|<m$. By induction the nilpotency class of $G/\gamma_{c-1}(G)$ is $m$-bounded, and hence so is the class of $G$.
\end{proof}

We now finish the proof of Theorem~\ref{t-l}. We have  $G/G'=[G/G',\varphi ]
\subseteq  {\mathscr L}_{G/G'}(\varphi)$ by Lemma~\ref{l3}(c). Since ${\mathscr L}_{G/G'}(\varphi)$ is the image of ${\mathscr L}_{G}(\varphi)$, we conclude that $|G/G'|\leqslant m$. Combined with the fact that the nilpotency class of $G$ is $m$-bounded by Lemma~\ref{l-lclass}, this implies that $|G|$ is $m$-bounded by Lemma~\ref{l-nilporder}.
\end{proof}

\begin{proof}[Proof of Corollary~\ref{c-lr}]
Recall that $\varphi$ is an automorphism of a group $G$, and we need to show that $|\langle {\mathscr L}(\varphi)\rangle|$ is bounded above  (a) in terms of  $|{\mathscr L}(\varphi)|$, and (b) in terms of $|{\mathscr R}_{G\langle \varphi\rangle}(\varphi)|$.

Clearly, the mapping $u\to [u,\varphi ]$ is surjective on the (minimal) left Engel sink ${\mathscr L}(\varphi)$. Therefore, $[\langle {\mathscr L}(\varphi)\rangle,\varphi ]=\langle {\mathscr L}(\varphi)\rangle$. Hence part (a) follows by Theorem~\ref{t-l}, and part (b) by Theorem~\ref{t-r}.
\end{proof}

\end{document}